\documentclass[11pt,reqno]{amsart}

 \usepackage[%
   colorlinks=true,
   urlcolor=rltblue,       
   filecolor=rltgreen,     
   citecolor=rltgreen,     
   linkcolor=rltred,       
   bookmarks=true,
   unicode,
   plainpages=false,
   ]{hyperref}
\usepackage{xcolor}
\definecolor{rltred}{rgb}{0.75,0,0}
\definecolor{rltgreen}{rgb}{0,0.5,0}
\definecolor{rltblue}{rgb}{0,0,0.75}

\usepackage{amsmath}
\usepackage{amssymb}
\usepackage{a4wide}
\usepackage{esint}
\usepackage{bm}

\usepackage[operator,rose,frak]{paper_diening}

\newcommand{\R}{\mathbb{R}}
\newcommand{\N}{\mathbb{N}}

\renewcommand{\div}{\operatorname{div}}
\newcommand{\Du}{\mathbf{Du}}

\newcommand{\di}{\partial_i}
\renewcommand{\dj}{\partial_j}

\newcommand{\eps}{\varepsilon}
\newcommand{\dx}{\,\textrm{d}\mathbf{x}}

\newcommand{\x}{\mathbf{x}}
\newcommand{\xk}{\mathbf{x}^k}
\newcommand{\y}{\mathbf{y}}

\newcommand{\xj}{x_1}
\newcommand{\xd}{x_2}

\newcommand{\xjk}{x_1^k}
\newcommand{\xdk}{x_2^k}
\newcommand{\xnk}{x_n^k}
\renewcommand{\u}{\mathbf{u}}
\newcommand{\w}{\mathbf{w}}
\newcommand{\g}{\mathbf{g}}
\newcommand{\h}{\mathbf{h}}

\newcommand{\nequiv}{\not\equiv}

\newcommand{\diam}{\operatorname{diam}}
\newcommand{\Tr}{\operatorname{Tr}}

\theoremstyle{plain}
\newtheorem{thm}{Theorem}
\newtheorem{lem}[thm]{Lemma}
\newtheorem{prop}[thm]{Proposition}
\newtheorem*{thm*}{Theorem}

\theoremstyle{definition}
\newtheorem{df}[thm]{Definition}

\theoremstyle{remark}
\newtheorem{rem}[thm]{Remark}

\newtoks\by
\newtoks\paper
\newtoks\book
\newtoks\jour
\newtoks\yr
\newtoks\pages
\newtoks\vol
\newtoks\publ
\newtoks\eds
\newtoks\proc
\def\ota{{\hbox{???}}}
\def\cLear{\by=\ota\paper=\ota\book=\ota\jour=\ota\yr=\ota
	\pages=\ota\vol=\ota\publ=\ota}
\def\endpaper{\textsc{\the\by}, \textit{\the\paper}.
	{\the\jour} \textbf{\the\vol} (\the\yr), \the\pages.\cLear}
\def\endbook{\textsc{\the\by}, \textit{\the\book}. \the\publ.\cLear}
\def\endprep{\textsc{\the\by}, \textit{\the\paper}. \the\jour.\cLear}
\def\endproc{\textsc{\the\by}, \textit{\the\paper}. \the\publ, \the\pages.\cLear}

\begin{document}

\title[Counterexample]{A counterexample related to the regularity\\ of the $p$-Stokes problem}
\author{Martin K\v repela} \address{Martin K\r repela, 
Institute of Applied Mathematics, University of  
  Freiburg, Eckerstr.~1, D-79104 Freiburg, Germany.}
  \email{martin.krepela@math.uni-freiburg.de} \author
{Michael R\r u\v zi\v cka{}} \address{Michael R\r u\v zi\v cka{},
  Institute of Applied Mathematics, University of 
  Freiburg, Eckerstr.~1, D-79104 Freiburg, Germany.}
\email{rose@mathematik.uni-freiburg.de}

\begin{abstract}
  In this paper we construct a solenoidal vector field $\bu$ belonging
  to $W^{2,q}(\Omega)\cap W^{1,s}_0(\Omega)$, $s \in (1,\infty)$,
  $q \in (1,n)$, such that $(1+|\bD\bu|)^{p-2}$,
  $p\in (1,2)\cup(2,\infty)$, does not belong to the Muckenhoupt class
  $A_\infty(\Omega)$. Thus, one cannot use the Korn inequality in
  weighted Lebesgue spaces to prove the natural regularity of the
  $p$-Stokes problem.
  \\[3mm]
  \\[3mm]
\end{abstract}
\date{\today}\maketitle

\centerline{\bf Dedicated to V.V.~Zhikov}

\section{Introduction}
\label{sec:intro}
The question of full natural regularity for the $p$-Stokes system
\begin{equation}
  \label{eq}
  \begin{aligned}
    -\divo \bS(\bD\bu) 
    +\nabla \pi&=\bff\qquad&&\text{in }\Omega,
    \\
    \divo\bu&=0\qquad&&\text{in }\Omega,
  \end{aligned}
\end{equation}
is still not completely solved in the three- and higher-dimensional situation. 
Here $\Omega\subset\setR^n$, $n \ge 2$, is a bounded domain. The
unknowns are the velocity vector field $\bu=(u_1,\ldots,u_n)^\top$ and
the scalar pressure $\pi$, while the external body force
$\bff=(f_1,\ldots,f_n)^\top$ is given. The extra stress tensor $\bfS$
depends only on $\bfD\bu:=\tfrac 12(\nabla\bu+\nabla\bu^\top)$, the
symmetric part of the velocity gradient $\nabla \bu$. Physical
interpretation and discussion of some non-Newtonian fluid models can
be found, e.g., in~\cite{bird}, \cite{mrr}, \cite{ma-ra-model}. The relevant example
which we have in mind is: 
\begin{equation}
  \label{eq:stress}
  \bfS(\bD \bu) = \mu (\delta+\abs{\bD \bu})^{p-2}
  \bD\bu\,, 
\end{equation}
with $p \in\, (1,\infty)$, $\delta > 0$, and $\mu >0$. We restrict
ourselves to the case $\mu=\delta=1$. However, the arguments in our
paper can be adapted such that the results also hold for $\delta>0$.  

Analogously to the regularity theory of nonlinear elliptic systems
(cf.~\cite{acerbi-fusco}, \cite{gia-mod-86}, \cite{giu1}), the
natural regularity for solutions of \eqref{eq} is 
\begin{equation}
  \bF (\bD\bu)\in W^{1,2}(\Omega)\,,\label{eq:reg}
\end{equation}
where
\begin{align}
  \label{eq:def_F}
  \bF(\bD\bu):= \big (1+\abs{\bD\bu} \big )^{\frac
    {p-2}{2}}{\bD\bu} \,.
\end{align}
Note that for $n=2$ the regularity problem for \eqref{eq}, and even for
the $p$-Navier--Stokes problem, is solved (cf.~\cite{KMS2}). In the
following, we concentrate on the three-dimensional situation,
i.e.,~$n=3$. Using the standard approach of divided differences, one can
show that weak solutions of \eqref{eq} satisfy
$\bF(\bD\bu) \in W^{1,2}_\loc (\Omega)$. From this it immediately follows
that \eqref{eq:reg} is
fulfilled for periodic boundary conditions. Unfortunately,
it is the only situation where this optimal
result is known. Despite various efforts (cf.~\cite{Eb00}, 
\cite{hugo-steady}, \cite{Eb2006a}, \cite{hugo-nonflat}, \cite{CG08},
\cite{hugo-troisi}, \cite{Crispo-2009}, \cite{hugo-thin},
\cite{hugo-thin-nonflat}, \cite{luigi-reg},
\cite{hugo-petr-rose}, \cite{br-reg-shearthin}), full regularity near
the boundary is still an open problem for the $p$-Stokes system \eqref{eq} completed with
zero Dirichlet boundary conditions. Many of the existing papers essentially prove 
that, under appropriate assumptions on the  regularity of the
boundary,  for some $q<2$ depending on $p$ there holds 
\begin{equation}\label{eq:part}
  \xi\,\partial _\btau \bF (\bD\bu)\in L^{2}(\Omega)\,,\quad   \xi\,\partial
  _\bnu \bF (\bD\bu)\in L^{q} ( \Omega)\,,
\end{equation}
where $\partial _\bnu$ is the local normal derivative,  $\partial _\btau$ a local tangential
derivative and $\xi \in C_0^\infty(\setR^3)$ an appropriate cut-off
function with support near the boundary $\partial \Omega$. The precise
definition of these quantities using a local description of the
boundary as a graph can be found, e.g.,~in \cite{br-reg-shearthin}. Note
that one can easily derive from \eqref{eq:part} that $\bu \in
W^{2,r}(\Omega)$, for some $r$ depending on $p$, i.e.,~regularity in
terms of Sobolev spaces. Some of the above mentioned papers prove
the Sobolev space regularity directly. All of these papers also contain certain
statements about the regularity of the pressure gradient. 

The reason for the non-optimality of the results near the boundary
lies in a subtle interplay between the dependence of the elliptic operator
on the symmetric part of the velocity gradient and the divergence
constraint resulting in the appearance of the pressure gradient. If
these difficulties are separated, one can prove better results. If
\eqref{eq} is considered without the divergence constraint and the
resulting pressure gradient, it is proved in \cite{SS00},
\cite{br-plasticity} that \eqref{eq:reg} holds for $p \in (1,2)$. If
\eqref{eq} is considered with $\bS$ depending on the full velocity
gradient $\nabla \bu$, it is proved in \cite{CM15} that
$\bu \in W^{2,r}(\setR^3)\cap W^{1,p}_0(\setR^3)$ for some $r>3$,
provided $p<2$ is very close to $2$. On the other hand, it is pointed out in
\cite{hugo-petr-rose}, \cite{br-reg-shearthin} that a Korn-type
inequality 
\begin{equation}\label{eq:korn}
  \int_\Omega (1+\abs{\bD \bu})^{p-2}|\nabla \partial _\tau\bu|^2 \dx \le c\,
  \int_\Omega (1+\abs{\bD \bu})^{p-2}|\bD \partial _\tau\bu|^2 \dx
\end{equation}
would yield the desired optimal results. However, the validity of
\eqref{eq:korn} is not known. It would be granted if a stronger assertion held true,
namely the Korn inequality in
the weighted Lebesgue space $L^2_\omega(\Omega)$ with 
$\omega = (1+\abs{\bD \bu})^{p-2}$. The validity of the
Korn inequality in Lebesgue spaces is proved in \cite{necas-66}, and in
weighted Lebesgue spaces with Muckenhoupt weights in \cite{john}. In the
weighted case, the proof is based on the continuity of the
maximal operator in weighted Lebesgue spaces. This, in turn, is
equivalent to the condition that the weight is a Muckenhoupt weight
(cf.~\cite{Mucken}).  The purpose of this paper is to construct a
solenoidal vector field $\bu \in W^{1,p}_0(\Omega)$ satisfying 
\eqref{eq:reg} but such that $(1+\abs{\bD \bu})^{p-2}$ does not
belong to any Muckenhoupt class for any $p\neq 2$ (cf.~Theorem \ref{14}).

\section{Counterexample}
Throughout the paper we use the standard notation for Lebesgue and
Sobolev spaces. We write $f \lesssim g$ if there exists a positive
constant $c$, depending only on irrelevant parameters, such that
$f \le c\,g$. Moreover, we write $f \approx g$ if both $f \lesssim g$
and $g\lesssim f$ are satisfied.

Let us start with a technical lemma which loosely follows the ideas of standard covering theorems of Vitali or Besicovitch (cf.~\cite[Section 1.5]{evans-gariepy}) It will ensure that the
constructed vector field will be a counterexample on every open subset
of $\Omega$. 
\begin{lem}\label{2}
  Let $n\ge 2$ and let $\Omega\subset\R^n$ be a~bounded or unbounded domain. Then there
  exists a~countable system of nonempty balls $\{E_l\}_{l\in\N}$
  contained in $\Omega$ and satisfying the following conditions:
  \begin{enumerate}
  \item[\rmfamily(i)] the set of all centers of the balls $E_l$ is
    dense in $\Omega$;
  \item[\rmfamily(ii)]
    $\sup_{l\in\N} \frac{\diam E_{l+1}}{\diam E_l} <\frac12$;
  \item[\rmfamily(iii)] if $k,l\in\N$ satisfy\footnote{We use the notation
    $\lfloor z\rfloor :=\max \{n\in \setN \fdg n\le z\}$.} $\lfloor\sqrt[3]{l}\rfloor\le k<l$, then
    $E_k\cap E_l = \varnothing$.
  \end{enumerate}
\end{lem}  

\begin{proof}
  Let $\{\x^i\}_{i\in\N}$ be a~dense sequence of points in
  $\Omega$. Choose a~sufficiently small $\varrho\in(0,\frac12)$ so
  that $B(\x^1,\varrho)\subset\Omega$ and
  $\Omega\setminus \overline{B(\x^1,\varrho)}\ne\varnothing$, and
  define $E_1:=B(\x^1,\varrho)$. Next, let $l\in\N$ and suppose that
  the balls $E_k$ are defined for all $k=1,\ldots,l$ so that the open
  set
  \[
    G_l:= \Omega\setminus \overline{\bigcup_{k=\lfloor \sqrt[3]{l+1}\rfloor}^l E_k}
  \]
  is not empty. Since $\{\x^i\}$ is dense in $\Omega$ and
  $G_l\subset\Omega$, there exists the minimal index $i\in\N$ such
  that $\x^i\in G_l$ and $\x^i$ is not a~center of any ball $E_k$,
  $k=1,\ldots,l$. The set $G_l$ is open, hence we can choose
  a~sufficiently small number $\delta\in(0,\varrho \diam E_l)$
  such that $B(\x^i,\delta)\subset G_l$ and
  $G_l\setminus \overline{B(\x^i,\delta)}\ne\varnothing$.  Then we
  define $E_{l+1}:=B(\x^i,\delta).$ Notice that
  \[
    G_{l+1}:= \Omega\setminus \overline{\bigcup_{k=\lfloor
        \sqrt[3]{l+2}\rfloor}^{l+1} E_k} \supset \Omega\setminus
    \overline{\bigcup_{k=\lfloor \sqrt[3]{l+1}\rfloor}^{l+1} E_k} =
    G_l\setminus\overline{E_{l+1}} \ne \varnothing. 
  \]
  The process continues by induction, generating the system
  $\{E_l\}_{l\in\N}$.
  
  From the construction it follows that
  $\diam E_{l+1}\le \varrho\, {\diam E_l}$ for all $l\in\N$. Since
  $\varrho<\frac12$, this gives the property (ii). As a~consequence,
  we also get $\diam E_l\downarrow 0$ as $l\to\infty$.
	
  In the $(l+1)$-st step of the construction, $E_{l+1}$ is defined so
  that $E_{l+1}\subset G_l$, hence $\{E_l\}_{l\in\N}$ has the property
  (iii).
	
  It remains to prove (i). Let $\x\in\Omega$ and $\eps>0$. Thanks to
  the density of $\{\x^i\}_{i\in\N}$ there exists an~$i\in\N$ such that
  $|\x-\x^i|<\eps$. If $\x^i$ is a~center of a~ball from the system
  $\{E_l\}$, we are finished. Suppose that it is not the case. We
  claim that this means that there exists $l_0=l_0(i) \in \setN$ such
  that for every $l\ge l_0$ we have $\x^i\notin G_l$.  Indeed, if
  $\x^i\in G_l$ and $\x^i$ is not chosen as the center of $E_{l+1}$,
  then there is an $\x^m\in G_l$ with $m<i$ which was not yet chosen
  as the center of a~ball $E_k$ and which becomes the center of
  $E_{l+1}$. This situation may occur only $(i-1)$ times before $i$
  becomes the smallest index such that $\x^i$ was not yet used as the
  center of some $E_k$. Hence, for all sufficiently large $l\in\N$ we
  get $\x^i\notin G_l$. It yields
  $\x^i \in \overline{\bigcup_{k=\lfloor \sqrt[3]{l+1}\rfloor}^l
    E_k}$,
  and thus $\x^i \in \overline{E_{k_0}}$ for some
  $k_0\ge\lfloor \sqrt[3]{l+1}\rfloor$. This and
  $\diam E_l\downarrow 0$, imply that for a~sufficiently large $l$ we
  get $|\x^i-\y|<\eps$, where $\y$ is the center of $E_{k_0}$. Therefore,
  $|\x-\y|\le|\x-\x^i|+|\x^i -\y|< 2\eps$ and the proof is finished.
\end{proof}

The following proposition is a~special case of \cite[Lemma 2.3]{AG92}. 

\begin{prop}\label{12}
  Assume that $2\le n < q <\infty$ and that $\Omega\subset\R^n$ is
  a~bounded $C^{1,1}$\,domain. Let
  $\g\in W^{1-\frac1q,q}(\partial\Omega)$ be a~vector field such that
  $\g\cdot \bm{\nu} = 0$ on $\partial\Omega$, where $\bm{\nu}$ denotes
  the outer unit normal of $\partial\Omega$. Then there exists
  a~vector field $\h\in W^{2,q}(\Omega)$ such that $\div \h = 0$ in
  $\Omega$, $\Tr_{\partial\Omega}\h \equiv 0$ and
  $\partial_{\bnu}\h=\g$ on $\partial\Omega$. 
\end{prop}	 

\begin{prop}\label{1}
  Assume that $2\le n < r <\infty$ and that $\Omega\subset\R^n$ is
  a~bounded $C^{1,1}$\,domain. Then there exists a~vector field
  $\w:\Omega\to\R^n$ such that
  $\w\in W^{2,r}(\Omega)\cap W^{1,\infty}_0(\Omega)$, both
  $|\nabla\w|\le 1$ and $\div \w=0$ hold in $\Omega$, and
  $\Tr_{\partial\Omega}\nabla\w\nequiv 0$.
\end{prop}	

\begin{proof}
  At any point $\x=(x_1,\ldots,x_n)^\top\in\partial\Omega$ define
  $\varphi(\x):=x_1$. Following \cite[Examples 4.2.2, 4.2.4]{baer},
  one can verify that $\varphi:\partial\Omega \to \R$ is a~$C^{1,1}$
  function and as such it satisfies that the gradient vector
  field\footnote{The gradient vector field
    $\operatorname {grad } \varphi$ of a function
    $\varphi:\partial\Omega \to \R$ is, at each point
    $\bx \in \partial \Omega$, the Riesz representative of the
    differential $d_\bx \varphi $ of $\varphi$ at $\bx $ with respect
    to the standard scalar product in $\setR^3$ restricted to the
    tangent plane $T_\bx \partial \Omega$. }
  $\operatorname {grad } \varphi\in C^{0,1}(\partial\Omega)$ and
  $\operatorname {grad } \varphi\cdot\bm{\nu}=0$ on
  $\partial\Omega$. (Notice that in \cite{baer} only the smooth
  situation for surfaces in $\setR^3$ is treated. However, the
  regularity conditions treated here may be obtained by a~careful tracking of the
  calculations in \cite{baer}, which can be easily generalized to
  the general case of a hypersurface in $\setR^n$.) Moreover, $\varphi$ is not constant on its
  domain. (This may be verified using the fact that $\Omega$ is
  a~(nonempty) bounded $C^{1,1}$\,domain.) Now we can take
  $0\nequiv\g:=\operatorname {grad } \varphi \in
  C^{0,1}(\partial\Omega)$, and get
  \begin{align*}
    \|\g\|_{W^{1-\frac1r,r}} 
    & =   \Bigg( \|\g\|_{L^r(\partial\Omega)} + \int_{\partial\Omega}\int_{\partial\Omega}
      \frac{|\g(\x)-\g(\y)|^{r}}{|\x-\y|^{r+n-2}} \,\textnormal{d}S(\x)\,\textnormal{d}S(\y)
      \Bigg)^\frac1r
    \\ 
    & \le \Bigg( \|\g\|_{L^\infty(\partial\Omega)} + C_\bg \,
      \int_{\partial\Omega}\int_{\partial\Omega} \frac1{|\x-\y|^{n-2}}
      \,\textnormal{d}S(\x)\,\textnormal{d}S(\y) \Bigg)^\frac1r. 
  \end{align*}
  A~standard calculation yields that the expression on the second line
  is bounded by a~constant depending on
  $\|\g\|_{L^\infty(\partial\Omega)}$ and the Lipschitz constant
  $C_\bg$ of $\g$. Hence, we have shown that
  $\g\in W^{1-\frac1r,r}(\partial\Omega)$.

  Proposition \ref{12} now provides a~function
  $\h\in W^{2,r}(\Omega)\cap W_0^{1,\infty}(\Omega)$ satisfying
  $\div\h= 0$ in $\Omega$ and $\partial_\bnu\h=\g\nequiv 0$ on
  $\partial\Omega$. (We have used the embedding
  $W^{2,r} \subset W^{1,\infty}$.) Then the function
  $\w:=\frac{\h}{\|\nabla\h\|_{\infty}}$ satisfies $|\nabla\w|\le 1$
  in $\Omega$. (This indeed can be stated at every point of $\Omega$
  since $\nabla\w \in W^{1,r}(\Omega)$ admits a~continuous
  representative therein.) Finally, since
  $\partial_\bnu\w = \frac{\g}{\|\nabla\w\|_\infty} \nequiv 0$ on
  $\partial\Omega$, it is verified that
  $\Tr_{\partial\Omega}\nabla\w \nequiv 0$ on $\partial\Omega$.
\end{proof} 

Before formulating the main result we recall when a
weight $\omega$ belongs to a~Muckenhoupt $A^p$ class.
\begin{df}
  Let $\Omega\subset\R^n$, $p\in(1,\infty)$, $p':=\frac{p}{p-1}$ and
  let $\omega$ be a~weight function, i.e., a locally integrable
  non-negative function on $\setR^n$. We say that $\omega $ satisfies
  the $A_p$ \emph{condition on} $\Omega$ if
  \[
    \sup_{B\subset\Omega} \ \fint_B \omega (\x)\dx \left( \fint_B
      \omega ^{1-p'}(\x)\dx \right)^{p-1} < \infty, 
  \]
  where $B\subset\Omega$ are balls. We denote by $A_p(\Omega)$ the class of all weights satisfying this condition. 
  Furthermore, we say that $\omega \in A_\infty(\Omega)$ if $\omega \in\bigcup_{p>1}A_p(\Omega)$. 
\end{df} 	 

\begin{thm}\label{14}
  Let $n\ge 2$ and $\Omega\subset\R^n$ be a~bounded
  $C^{1,1}$\,domain. Then there exists a~vector field
  $\u:\Omega\to \R^n$ such that:
  \begin{enumerate}
  \item[\rmfamily(i)] $\u\in W^{2,q}(\Omega)\cap W^{1,s}_0(\Omega)$
    for any $1<q<n$ and $1<s<\infty$;
  \item[\rmfamily(ii)] $\Tr_{\partial\Omega}\nabla\u \nequiv 0;$
  \item[\rmfamily(iii)] $\div \u \equiv 0$ in $\Omega$;
  \item[\rmfamily(iv)]
    $(1+|\Du|)^{p-2} \notin A_\infty(\Omega_0)$ for
    any $p\in(1,\infty)$, $p\ne 2$ and any open set
    $\Omega_0\subset\Omega$.
  \end{enumerate}			
\end{thm}	

\begin{rem}
  In the assertion (iv), the function $(1+|\Du|)^{p-2}$ is understood
  as extended by zero to the whole $\R^n$ so that it corresponds to
  the formal definition of a~weight.  Notice that we intentionally
  prove the stronger property
  $(1+|\Du|)^{p-2} \notin A_\infty(\Omega_0)$ for any open
  $\Omega_0\subset\Omega$, to exclude the possibility that a
  localisation of the weight near the boundary would be a~Muckenhoupt
  weight, which in turn would imply the validity of a localized
  version of~\eqref{eq:korn}.
\end{rem}

\begin{proof}[Proof of Theorem \ref{14}]
  By Lemma \ref{2}, we find a~system of balls $\{E_k\}_{k\in\N}$,
  $E_k=B(\xk,2r_k)$ such that 
    $\bigcup_{k\in\N} E_k \subset\Omega$;
    $\{\xk\}_{k\in\N}$ is dense in $\Omega$;
    there exists an~$\varrho\in(0,\frac12)$ such that $r_{k+1}\le \varrho r_k$ for all $k\in\N$;
    if $k,l\in\N$ satisfy $\lfloor\sqrt[3]{l}\rfloor \le k<l$, then $E_k\cap E_l = \varnothing$.
  Without loss of generality we can suppose that $r_1\le\varrho$. Let
  us denote $B_k:=B(\xk,{r_k})$ for each $k\in\N$, and consider that
  each point $\xk$ is represented by $\xk=( \xjk,\ldots,\xnk)^\top$.

  For every $k\in\N$, $3\le m \le n$ and $\x\in\overline \Omega$ set
  \begin{align*}
    u_1^k(\x) &:= \frac{k}{r_k}\chi_{B_k}(\x) (\xd-\xdk) \left( \frac{r_k}2 + \frac{|\x-\xk|^2}{2r_k} -|\x-\xk| \right), \\
    u_2^k(\x) &:= - \frac{k}{r_k}\chi_{B_k}(\x) (\xj-\xjk) \left( \frac{r_k}2 + \frac{|\x-\xk|^2}{2r_k} -|\x-\xk| \right), \\
    u_m^k(\x) &:= 0,
  \end{align*}
  and define the vector field $\u^k:\overline \Omega\to\R^n$ with
  compact support in $\Omega$ by $\u^k:=( u^k_1,\ldots,u_n^k)^\top$.
	
  Let $k\in\N$, and $i,j\in\{1,\ldots,n\}$.  The classical partial
  derivatives of $\bu^k$ can be computed as 
  \begin{align*} 
    \di u_1^k(\x) & = \frac{k}{r_k}\chi_{B_k}(\x) \Bigg[ \delta_{2i} \left( \frac{r_k}2 + \frac{|\x-\xk|^2}{2r_k} - |\x-\xk|\right) \\
                  & \qquad + (x_2-x^k_2)(x_i-x^k_i) \left( \frac1{r_k} - \frac{1}{|\x-\xk|} \right) \Bigg]\,, \qquad \x\in\Omega\setminus \{\xk\}\,,\\
    \di u_2^k(\x) & = - \frac{k}{r_k}\chi_{B_k}(\x) \Bigg[ \delta_{1i} \left( \frac{r_k}2 + \frac{|\x-\xk|^2}{2r_k} - |\x-\xk|\right) \\
                  & \qquad + (x_1-x^k_1)(x_i-x^k_i) \left( \frac1{r_k} - \frac{1}{|\x-\xk|} \right) \Bigg]\,, \qquad \x\in\Omega\setminus \{\xk\}\,,\\
    \dj\di u_1^k(\x) & = \frac{k}{r_k}\chi_{B_k}(\x) \Bigg[ \left(
     \delta_{2i} (x_j-x^k_j) + \delta_{2j} (x_i-x^k_i) + \delta_{ji}(x_2-x^k_2)
     \right)\left(\frac1{r_k} - \frac{1}{|\x-\xk|} \right) \\ 			 
                  & \qquad +
                    \frac{(x_2-x^k_2)(x_i-x^k_i)(x_j-x^k_j)}{|\x-\xk|^3}\Bigg]\,,
    \qquad \x\in\Omega\setminus\big( \{\xk\}\cup \partial B_k\big)\,,\\
    \dj\di u_2^k(\x) & = - \frac{k}{r_k}\chi_{B_k}(\x) \Bigg[ \left( \delta_{1i} (x_j-x^k_j) + \delta_{1j} (x_i-x^k_i) + \delta_{ji}(x_1-x^k_1) \right)\left(\frac1{r_k} - \frac{1}{|\x-\xk|} \right) \\ 			
                  & \qquad + \frac{(x_1-x^k_1)(x_i-x^k_i)(x_j-x^k_j)}{|\x-\xk|^3}\Bigg]\,, \qquad \x\in\Omega\setminus\big( \{\xk\}\cup \partial B_k\big)\,,\\
    \di u_m^k(\x) & = \dj\di u_m^k(\x) \equiv 0, \qquad 3\le m \le
                    n\,, \qquad \bx \in \Omega\,.
  \end{align*}
  Let $i,j,m\in\{1,\ldots,n\}$. 
  The function $u^k_m$ is continuous in $\overline \Omega$, the
  function $\di u_m^k$ is continuous in ${\Omega}\setminus\{\xk\}$,
  and $\dj\di u_m^k$ is continuous in
  ${\Omega}\setminus\big(\{\xk\}\cup \partial B_k\big)$.  Recalling
  that if $\x\in B_k$, then $|x_i-x^k_i|\le |\x-\xk|\le r_{k}$, we
  obtain the following estimates:
  \begin{alignat}{2} 
    |u_m^k(\x)| & \le kr_k \chi_{B_k}(\x) \,, \qquad &&\x\in\overline \Omega\,, \label{3}\\
    |\di u_m^k(\x)| & \le 2k \chi_{B_k}(\x) \,, \qquad &&\x\in\Omega\setminus \{\xk\}\,, \label{4} \\
    |\dj\di u_m^k(\x)| & \le \frac{4k}{r_k} \chi_{B_k}(\x) \,, \qquad &&\x\in\Omega\setminus\big( \{\xk\}\cup \partial B_k\big)\,. \label{5}
  \end{alignat}
  The function $\di u_m^k$ belongs to $L^\infty(\Omega)$.  Hence,
  using a~classical argument, based on the approximation of $\Omega$ by
  \[
    \Omega\setminus B(\xk,\eps), \quad \eps\to 0_+, 
  \]
  and the continuity of $u_m^k$ in ${\Omega}$, one can show that $\di u_m^k$ is the weak derivative of $u^k_m$ on
  $\Omega$.  Next, also $\dj\di u_m^k $ belongs to $L^\infty
  (\Omega)$. Again, an argument using the approximation of $\Omega$ by
  \[
    \Omega\setminus\big( B(\xk,\eps) \cup \{\x\in\Omega \big|\
    r_k-\eps \le |\x-\xk| \le r_k+\eps \}\big), \quad \eps\to 0_+, 
  \]
  and the continuity of $\di u_m^k$ in ${\Omega}\setminus\{\xk\}$, shows
  that the function $\dj\di u_m^k$ is the second-order weak derivative
  of $u^k_m$ on $\Omega$.
  
  Let $l\in\N$. Using \eqref{3} and the assumption $r_1\le\varrho$, we get
  \[
    \sum_{k=1}^l \sup_{\x\in\overline \Omega}|u_m^k(\x)| \le \sum_{k=1}^l kr_k \le \sum_{k=1}^l k\varrho^k < \sum_{k\in\N} k\varrho^k <\infty.
  \]
  Hence, there exists a~function $\u^0\in C(\overline \Omega)$ such that
  $\sum_{k=1}^l \u^k\to\u^0$ absolutely uniformly on $\overline \Omega$ as
  $l\to\infty$.
	
  Let $s>1$. Applying \eqref{4}, we have 
  \begin{align*}
     \sum_{k=1}^l \left( \int_\Omega |\di u_m^k(\x)|^s \dx \right)^\frac1s &\le 2 \sum_{k=1}^l k \left( \int_\Omega  |\chi_{B_k}(\x)| \dx \right)^\frac1s  \\
    & = 2 \sum_{k=1}^l k|B_k|^\frac1s \lesssim \sum_{k=1}^l kr_k^{\frac ns} \le \sum_{k=1}^l k\varrho^{\frac{nk}s} < \sum_{k\in\N} k\varrho^{\frac{nk}s} < \infty.
  \end{align*}
  For every $l\in\N$ we have
  $\sum_{k=1}^l \u^k \in W^{1,s}_0(\Omega)$, and
  \[
    \sum_{k\in\N} \|\u^k\|_{W^{1,s}_0(\Omega)} \approx \sum_{k\in\N} \|\nabla\u^k\|_{L^s(\Omega)} \lesssim \sum_{k\in\N} k\varrho^{\frac{nk}s} < \infty.
  \]
  Therefore, $\big\{\sum_{k=1}^l\u^k\big\}_{l\in\N}$ is a~Cauchy
  sequence in $W^{1,s}_0(\Omega)$. Since
  $\sum_{k=1}^l \u^k \xrightarrow{l\to\infty} \u^0$ pointwise in
  $\Omega$, it is also true that
  $\sum_{k=1}^l \u^k\xrightarrow{l\to\infty} \u^0$ in
  $W^{1,s}_0(\Omega)$. Moreover, $\div \u^0\equiv 0$ since
  $\div \sum_{k=1}^l \u^k\equiv 0$ for all $l\in\N$.
	
  Next, let $q\in(1,n)$. Using \eqref{5}, we get
  \begin{align*}
    \sum_{k=1}^l \left( \int_\Omega |\dj\di u_m(\x)|^q \dx
    \right)^\frac1q 
    & \le 4 \sum_{k=1}^l \frac{k}{r_k} \left( \int_\Omega  |\chi_{B_k}(\x)| \dx
      \right)^\frac1q = 4 \sum_{k=1}^l \frac{k|B_k|^\frac 1q}{r_k} 
    \\ 
    & \lesssim \sum_{k=1}^l kr^{\frac nq -1}_k \le \sum_{k=1}^l
      k\varrho^{\left(\frac nq -1\right)k} < \sum_{k\in\N}
      k\varrho^{\left(\frac nq -1\right)k} < \infty. 
  \end{align*}
  For every $l\in\N$ one has
  $\sum_{k=1}^l \u^k \in W^{2,q}_0(\Omega)$, and
  \[
    \sum_{k\in\N} \|\u^k\|_{W^{2,q}_0(\Omega)} \approx \sum_{k\in\N} \|\nabla^2\u^k\|_{L^q(\Omega)} \lesssim \sum_{k\in\N} k\varrho^{\left(\frac nq -1\right)k} < \infty.
  \]
  It implies that $\big\{\sum_{k=1}^l\u^k\big\}_{l\in\N}$ is a~Cauchy
  sequence in $W^{2,q}_0(\Omega)$. Taking the pointwise convergence
  into account, we obtain that
  $\sum_{k=1}^l \u^k \xrightarrow{l\to\infty} \u^0$ in
  $W^{2,q}_0(\Omega)$.
	
  We have shown that
  $\mathbf{u}^0 \in C(\overline \Omega) \cap W^{1,s}_0(\Omega) \cap
  W^{2,q}_0(\Omega)$
  for any $s\in(1,\infty)$ and $q\in(1,n)$. Now consider the field
  $\w\in W^{2,r}(\Omega)\cap W^{1,\infty}_0(\Omega) \subset
  W^{2,q}(\Omega)\cap W^{1,\infty}_0(\Omega)$ (where $r>n$)
  granted by Proposition \ref{1}, and define $\u:=\u^0 +
  \w$. Obviously, the field $\u$ satisfies (i), (ii) and (iii).
	
  It remains to verify (iv). There exists a~constant $\eps\in(0,1)$ dependent only on $n$ and such that for any $l\in\N$ one has 
  \[
    \left| G_l \right| \ge \frac23 |B_l|,
  \]
  where
  \[
    G_l := \left\{ \x\in B_l\fdg  |x_1-x^l_1|\ge r_{l}\eps,\ |x_2-x^l_2|\ge r_{l}\eps,\ |\x-\x^l|\le r_{l}(1-\eps) \right\}.
  \]
  For each $l\in\N$ define
  $M_l := G_l\setminus \bigcup_{k=l+1}^\infty B_k$. Recall that
  $\u^0=\sum_{k\in\N} \u^k$ (convergence in $W^{1,s}$) and
  $\u^k = \u^k\chi_{B_k}$ for every $k\in\N$. Hence,
  \begin{equation}\label{6}
    \Du^0(\x)= \sum_{k=1}^l \Du^k(\x)\chi_{B_k}(\x)\quad  \textnormal{for a.e.\ } \x \in \Omega\setminus \bigcup_{k=l+1}^\infty B_k.
  \end{equation}
  Let $l\in\N$, $l\ge 2$. From the properties of the system $\{E_k\}$
  it follows that $B_k\cap M_l \subset E_k\cap E_l = \varnothing$ for
  all $k,l\in\N$ such that $\sqrt[3]{l}<k<l$. Together with \eqref{6}
  it gives
  \begin{equation}\label{11}
    \Du^0(\x) = \Du^l(\x) + \sum_{k=1}^{\lfloor\sqrt[3]{l}\rfloor} \Du^k(\x) \quad \textnormal{for\ a.e.\ } \x\in M_l.
  \end{equation}
  Therefore, we may write
  \begin{align*}
    |\Du^0(\x)| & = \Big| \Du^l(\x) +
    \sum_{k=1}^{\lfloor\sqrt[3]{l}\rfloor} \Du^k(\x)
    \Big| \ge |\Du^l(\x)| - \Big|
   \sum_{k=1}^{\lfloor\sqrt[3]{l}\rfloor} \Du^k(\x)\Big| 
    \\ 
    & \ge |\Du^l(\x)| - \sum_{k=1}^{\lfloor\sqrt[3]{l}\rfloor}
      |\nabla\u^k(\x)| \ge |\partial_1 u_1^l(\x)| - n
      \sum_{k=1}^{\lfloor\sqrt[3]{l}\rfloor} \max_{1\le
      i,j \le n}|\di u_j^k(\x)| . 
  \end{align*}
  for a.e.~$\x\in M_l$. Since $\x\in M_l\subset G_l$, we may write
  \[
    |\partial_1 u_1^l(\x)|  = \frac{l}{r_l} |x_2-x^l_2||x_1-x^l_1|\left( \frac{1}{|\x-\x^l|} - \frac1{r_l} \right)
    \ge \eps^2 l r_l  \left( \frac1{r_l(1-\eps)} - \frac1{r_l} \right) 						  	
    \ge \frac{\eps^3 l}{1-\eps}.
  \]
  Using this estimate and \eqref{4}, we obtain
  \begin{equation}\label{odhad}
    |\Du^0(\x)| \ge |\partial_1u_1^l(\x)| - n
    \sum_{k=1}^{\lfloor\sqrt[3]{l}\rfloor} \max_{1\le i \le n}|\di
    u^k(\x)| \ge \frac{\eps^3 l}{1-\eps} - 2 n
    \sum_{k=1}^{\lfloor\sqrt[3]{l}\rfloor} k \ge \frac{\eps^3 l}{1-\eps}
    - 2 nl^{\frac23} 
  \end{equation}		
  for a.e.~$\x\in M_l$. Moreover, since $\varrho^n < 2^{-n} \le \frac14$, we have
  \[
    \frac{ \Big|\bigcup_{k=l+1}^\infty B_k \Big|}{|B_l|} \le \sum_{k=l+1}^\infty \frac{ |B_k| }{|B_l|} = \sum_{k=l+1}^\infty \varrho^{n(k-l)} \le \frac13. 
  \]
  This in turn yields
  \[
    |M_l| \ge |G_l|-\Big| \bigcup_{k=l+1}^\infty B_k \Big| \ge \frac13 |B_l| = \frac1{3\cdot2^n} |E_l|.
  \]
  Similarly, we can also derive the estimate
  $\Big|\bigcup_{k=l}^\infty B_k \Big|\le \frac43 |B_l|$, from which
  it follows that
  \begin{equation}\label{7}
    \Big| E_l \setminus \bigcup_{k=l}^\infty B_k \Big| \ge |E_l| - \frac43 |B_l| = \left( 1 - \frac1{3\cdot 2^{n-2}} \right)|E_l|.
  \end{equation}
  Since $\u^l=\u^l\chi_{B_l}$, by using \eqref{11} we obtain the
  estimate
  \begin{equation}\label{8}
    |\Du^0(\x)| =  \sum_{k=1}^{\lfloor\sqrt[3]{l}\rfloor} | \Du^k(\x)
    | \le n \sum_{k=1}^{\lfloor\sqrt[3]{l}\rfloor} \max_{1\le i,j\le
      n} | \partial_i u_j^k(\x) | \le 2n
    \sum_{k=1}^{\lfloor\sqrt[3]{l}\rfloor} k \le 2nl^{\frac23},   
  \end{equation}
  valid for a.e.~$\x\in E_l\setminus\bigcup_{k=l}^\infty B_k$.
  Furthermore, considering that $|\nabla\w|\le 1$ a.e.~in $\Omega$, we
  also have the relation
  \[
    1 + |\Du| = 1 + |\Du^0 + \mathbf{Dw}| \ge 1 + |\Du^0| - |\mathbf{Dw}| \ge 1 + |\Du^0| - |\nabla \w| \ge |\Du^0|
  \]
  a.e.~in $\Omega$. Now we have prepared everything to show that $(1
  +|\bD\bu|)^{p-2}$ does not satisfy the $A_\alpha$ condition for any
  $\alpha>1$. First, for $p>2$ we have
  \begin{align*}
    \int_{E_l}(1+|\Du(\x)|)^{p-2} \dx & \ge \int_{E_l}|\Du^0(\x)|^{p-2} \dx \\
                                      & \ge \int_{M_l}|\Du^0(\x)|^{p-2} \dx \\
                                      & \ge \left( \frac{\eps^3 l}{1-\eps} - 2 nl^{\frac23} \right)^{p-2} |M_l| \\
                                      & \ge \frac1{3\cdot2^n}	\left( \frac{\eps^3 l}{1-\eps} - 2 nl^{\frac23} \right)^{p-2} |E_l|.			  	
  \end{align*}
  Hence, 
  \begin{equation}\label{9}
    \fint_{E_l}(1+|\Du(\x)|)^{p-2} \dx \ge C_1 (C_2l-l^{\frac23})^{p-2},
  \end{equation}	
  where $C_1, C_2$ are positive constants depending only on $n$ and
  $p$. Next, applying \eqref{8} and \eqref{9} we obtain for $\alpha >1$
  \begin{align*}
    \int_{E_l}(1+|\Du(\x)|)^{(p-2)(1-\alpha')} \dx 
    & \ge \int_{E_l\setminus \bigcup_{k=l}^\infty
      B_k}(1+|\Du(\x)|)^{(p-2)(1-\alpha')}      \dx  
    \\ 
    & \ge \int_{E_l\setminus \bigcup_{k=l}^\infty B_k}(1+|\nabla\w(\x)|+|\Du^0(\x)|)^{(p-2)(1-\alpha')} \dx  \\
    & \ge \big(2 + 2nl^{\frac23}\big)^{(p-2)(1-\alpha')} \Big| E_l\setminus \bigcup_{k=l}^\infty B_k \Big|  \\
    & \ge \big(2 + 2nl^{\frac23}\big)^{(p-2)(1-\alpha')} \left( 1 - \frac1{3\cdot 2^{n-2}} \right)|E_l|.
  \end{align*}
  Thus,
  \[
    \left( \fint_{E_l} (1+|\Du(\x)|)^{(p-2)(1-\alpha')} \dx
    \right)^{\alpha-1} \ge C_3^{\alpha-1}(C_4 + l^{\frac23})^{2-p}, 
  \]
  where $C_3, C_4$ are positive constants depending only on $n$ and
  $p$. By combining the obtained estimates we get
  \begin{equation}\label{10}
    \fint_{E_l}(1+|\Du(\bx)|)^{p-2}\dx  \left( \fint_{E_l}
      (1+|\Du(\bx)|)^{(p-2)(1-\alpha')} \dx\right)^{\alpha-1} \hspace{-5pt} \ge
    \frac{C_1}{C_3^{\alpha-1}} \! \left( \frac{C_2l-l^{\frac23}}{C_4 +
        l^{\frac23}} \right)^{p-2} \hspace{-5pt}
    \xrightarrow{l\to\infty} \infty.  
  \end{equation}
  Let $\Omega_0\subset\Omega$ be an~open set. Since the set of centers
  of the balls $E_l$ is dense in $\Omega$ and
  $\diam E_l \downarrow 0$, there exists an~infinite subsequence
  $\{E_{l_k}\}_{k\in\N}$ such that $E_{l_k}\subset\Omega_0$ for all
  $k\in\N$. The property \eqref{10} then implies that the function
  $(1+|\Du|)^{p-2}$ does not satisfy the $A_\alpha$ condition on
  $\Omega_0$ for any $p>2$ and $\alpha>1$.
	
  Now, let $p\in(1,2)$ and $\alpha>1$. Define
  $p_0:=\frac{p-2}{1-\alpha}+2$. Since $p_0>2$ now, from \eqref{10} we
  get
  \begin{align*}
    & \fint_{E_l}(1+|\Du(\x)|)^{p-2} \dx\ \left( \fint_{E_l} (1+|\Du(\x)|)^{(p-2)(1-\alpha')} \dx \right)^{\alpha-1} \\
    & = \left( \left( \fint_{E_l}(1+|\Du(\x)|)^{(p_0-2)(1-\alpha)} \dx
      \right)^{\alpha'-1}  \fint_{E_l} (1+|\Du(\x)|)^{p_0-2} \dx
      \right)^{\alpha-1} \xrightarrow{l\to\infty} \infty.  
  \end{align*}
  By the same argument as in the previous part of the proof, we
  conclude that the function $(1+|\Du|)^{p-2}$ does not belong to the
  $A_{\alpha'}(\Omega_0)$ class for any $p\in(1,2)$, $\alpha>1$ and
  any open set $\Omega_0\subset\Omega$. This completes the proof.
\end{proof}	

\begin{rem}\label{R1}
	Estimates \eqref{odhad} and \eqref{8} should clarify the role of the bound $\lfloor\sqrt[3]{l}\rfloor$ from Lemma \ref{2}. This bound was chosen with the purpose of making the right-hand side in the final estimate \eqref{10} divergent.
\end{rem}

\begin{rem}\label{13}
  Notice that the function $\u$ from Theorem \ref{14} satisfies
  \eqref{eq:reg}. One easily sees that $\bF(\bD\bu) \in L^2(\Omega)$
  is equivalent to $\bu \in W^{1,p}(\Omega)$. This is satisfied since
  $\bu \in W^{1,s}(\Omega)$, $1<s<\infty$.  Moreover, one can show
  (cf.~\cite{bdr-7-5}) that
  \begin{equation*}
    \abs{\nabla \bF(\bD\bu)}^2 \approx  (1+|\Du|)^{p-2}|\nabla
    \bD\u|^2\approx  (1+|\Du|)^{p-2}|\nabla^2\u|^2\,, 
  \end{equation*}
  where we also used the algebraic identity 
  \begin{align*}
    \partial_j\partial_kv^i= \partial _j D_{ik}\bv +\partial _k D_{ij}\bv -\partial _i D_{jk}\bv \,,
  \end{align*}
  valid for sufficiently smooth vector fields $\bv$. Thus, it is
  sufficient to verify that
  \[
    \int_\Omega (1+|\Du|)^{p-2}|\nabla^2\u|^2 \dx< \infty
  \]
  for any $p\in (1,\infty)$. If $1<p\le 2$, then
  \[
    \int_\Omega (1+|\Du|)^{p-2}|\nabla^2\u|^2  \dx\le \int_\Omega |\nabla^2\u|^2  \dx< \infty,
  \]
  since $\nabla^2\u\in L^q(\Omega)$ for any $1<q<n$, in particular
  $q=2$.  If $p>2$, then
  \[
    \int_\Omega \! (1+|\Du|)^{p-2}|\nabla^2\u|^2  \dx \le 2^p \!\int_\Omega \!
    |\nabla^2\u|^2  \dx+ 2^p \! \left( \int_\Omega \!|\nabla^2\u|^\frac52  \dx
    \right)^\frac45 \!\!\left( \int_\Omega \!|\Du|^{5(p-2)}  \dx\right)^\frac15 \!<
    \infty,
  \]
  since $\u\in W^{2,q}(\Omega)$ for any $1<q<n$. More precisely, one
  can choose the $q\in \big[\frac52,n\big)$ so that
  $\frac{nq}{n-q} \ge 5(p-2)$ and thus
  $\nabla^2\u \in L^{\frac52}(\Omega)$ and
  $\Du \in W^{1,q}(\Omega) \subset L^{\frac{nq}{n-q}}(\Omega) \subset
  L^{5(p-2)}(\Omega)$.
\end{rem}		

\begin{rem}
  Let $\bu$ be the vector field constructed in Theorem \ref{14} in the
  case $n=3$. Setting $\ff:= - \divo \bS(\bD \bu)$, where $\bS$ is
  given by \eqref{eq:stress} with $\delta=1$, we see that $\bu$ and
  $\pi\equiv 0$ are a weak solution of the $p$-Stokes system
  \eqref{eq} with zero Dirichlet boundary condition for any
  $p \in (1,\infty)$. From Remark \ref{13} we know that $\bu$
  possesses the natural regularity \eqref{eq:reg}. Computations
  similar to the ones in Remark \ref{13} show that
  $\ff \in L^2(\Omega)$ for any $p \in (1,\infty)$. For $p \ge 2$,
  this regularity of the right-hand side $\ff$ is exactly the requirement
  posed in the papers dealing with the regularity of \eqref{eq} which
  were mentioned in the introduction. For $p\in (1,2]$, the usual
  requirement on the right-hand side of \eqref{eq} in the regularity
  investigations is $\ff \in L^{p'}(\Omega)$. This is fulfilled for
  the above defined $\ff$ for $p \in \big(\frac32,2\big]$.
\end{rem}
\begin{rem}
  Using the same notation as in the previous remark, one sees that
  $\bu$ is for any $p \in (1,\infty)$ also a solution of \eqref{eq}
  without the divergence constraint and the resulting pressure
  gradient, but equipped with zero Dirichlet boundary condition. From
  Remark \ref{13} we know that $\bu$ possesses the natural regularity
  \eqref{eq:reg}. However, for this modified problem this regularity
  property is proved in \cite{SS00}, \cite{br-plasticity} for any weak
  solution and any right-hand side $\ff \in L^{p'}(\Omega)$.
\end{rem}
	
\def\cprime{$'$} \def\cprime{$'$} \def\cprime{$'$}
\ifx\undefined\bysame
\newcommand{\bysame}{\leavevmode\hbox to3em{\hrulefill}\,}
\fi

\end{document}